\documentclass[a4paper, twosided]{article}
\usepackage[utf8]{inputenc}
\usepackage[english]{babel}
\usepackage{amsmath,amsthm}
\usepackage{amsfonts}
\usepackage{amssymb}
\usepackage{graphicx}
\usepackage[]{geometry}
\usepackage{tikz}
\usepackage{tikz-3dplot,pgfplots}
\usepackage{color}
\usetikzlibrary{patterns,decorations.pathreplacing}
\usepackage{enumerate}

 \newtheorem{theorem}{Theorem}[section]

   \newtheorem{corollary}[theorem]{Corollary}
  \newtheorem{lemma}[theorem]{Lemma}
     
    \newtheorem{conjecture}[theorem]{Conjecture} 
    
        \newtheorem{innercustomthm}{Theorem}
\newenvironment{theorem*}[1]
  {\renewcommand\theinnercustomthm{#1}\innercustomthm}
  {\endinnercustomthm}
  
  \newtheorem{innercustomlem}{Lemma}
\newenvironment{lemma*}[1]
  {\renewcommand\theinnercustomlem{#1}\innercustomlem}
  {\endinnercustomlem}
    
    \theoremstyle{definition}
\newtheorem{definition}{Definition}[section]

\theoremstyle{remark}    
    \newtheorem{remark}{Remark}[section]

  \newcommand{\G}{\Gamma}
    \newcommand{\AG}{\ensuremath{{A_{S}}}}

        \newcommand{\WG}{\ensuremath{{W_{S}}}}
          \newcommand{\CPG}{\ensuremath{{\mathcal{CP}_{A_S}}}}

          \newcommand{\Cayley}{\ensuremath{{\Gamma(A_S,\mathcal{N})}}}
  
\tikzset{vertex/.style={circle, draw, fill=black!50},inner sep=0pt, minimum width=4pt}

\title{Parabolic subgroups in FC type Artin groups}
\author{Rose Morris-Wright}

\begin{document}

\maketitle 

\begin{abstract}
Parabolic subgroups are the building blocks of Artin groups. This paper extends previous results, known only for parabolic subgroups of finite type Artin groups, to parabolic subgroups of FC type Artin groups. We show that the class of finite type parabolic subgroups is closed under intersection. We also study an analog of the curve complex for mapping class group constructed using parabolic subgroups. We extend the construction of the complex of parabolic subgroups to FC type Artin groups. We show that this simplicial complex is, in most cases, infinite diameter and conjecture that it is $\delta-$hyperbolic.
\end{abstract}

\section{Introduction}

An Artin group, $\AG$, is a group with generating set $S$ and relations similar to those of a braid group.
 A \textit{standard parabolic subgroup}, $A_T$, is a subgroup generated by a subset  $T\subset S$  and more generally a \textit{parabolic subgroup} is a conjugate of a standard parabolic subgroup, $gA_T g^{-1}$ for some $g\in \AG$. Parabolic subgroups can be thought of as the building blocks of Artin groups and so studying these subgroups and how they fit together within the overarching Artin group can provide valuable insights about Artin groups. 

Just as each braid group is associated to a symmetric group, each Artin group is associated to a Coxeter group. If the Coxeter group associated to a given Artin group is finite then the Artin group is said to be of \textit{finite type}. Finite type Artin groups are the most well understood Artin groups while 
infinite type Artin groups are much more mysterious.  The Deligne complex is a geometric construction that allows some results on finite type Artin groups to be generalized to infinite type Artin groups. The Deligne complex is most useful when it is CAT(0), and this occurs precisely in a special subclass of Artin groups known as FC type Artin groups.
 This paper will look at some facts about parabolic subgroups of finite type Artin groups and generalize them to FC type Artin groups. 

Some algebraic properties of parabolic subgroups of Artin groups are known. In \cite{VanderLek1983}, van der Lek shows that each parabolic subgroup is itself isomorphic to an Artin group.  In \cite{Paris1997}, Paris restricts to parabolic subgroups within finite type Artin groups and gives results about their centralizers and quasi-centralizers, including an algorithm to determine if two given parabolic subgroups are conjugate. In \cite{Godelle2003} and \cite{Godelle2007}, Godelle expands upon Paris's results and generalizes them to the case where the Deligne complex has a CAT(0) metric. Godelle explicitly computes the normalizer for a parabolic subgroup and shows that if $P$ and $Q$ are parabolic subgroups of an FC type Artin group and $Q\subset P$ then $Q$ is also a parabolic subgroup of $P$. Yet even this seemingly simple result is unknown for more general Artin groups. 

Another seemingly simple question is whether the set of parabolic subgroups is closed under intersection. In \cite{VanderLek1983}, van der Lek also shows that if $R,T$ are subsets of $S$ then $A_R\cap A_T=A_{R\cap T}$, but a more general result about the intersections of nonstandard parabolic subgroups has been much harder to establish. In \cite{Cumplido2017}, Cumplido et al show that the intersection of two parabolic subgroups in a finite type Artin group is also a parabolic subgroup. Their proof relies heavily on the Garside structure of the finite type Artin group, an algebraic construction that is difficult to generalize beyond the finite type case. 
This paper uses geometric techniques to generalize their result. 

\begin{theorem*}{\ref{Thrm:Intersections}}
Suppose that $\AG$ is an  FC type Artin group. Given two finite type parabolic subgroups, $P$ and $Q$, the intersection $P\cap Q$ is also a finite type parabolic subgroup. 
\end{theorem*}



Another application of parabolic subgroups is to create a simplicial complex for Artin groups which is analogous to the curve complex for mapping class groups. By thinking of a braid group as the mapping class group of the punctured disc, we observe a bijection between the simple closed curves in the disc and proper irreducible parabolic subgroups of the braid group. Simple closed curves in the disc are used to construct the curve complex for the mapping class group. Analogously, we may use parabolic subgroups to construct a simplicial complex for a given Artin group. If the Artin group is a braid group, this simplicial complex will be isomorphic to the curve complex of the punctured disc, with isomorphism via the bijection between simple closed curves and proper irreducible parabolic subgroups. 

In the case of the mapping class group of a surface, the associated curve complex is (in all but a few sporadic cases) connected, infinite diameter, and $\delta-$hyperbolic \cite{Masur1998}. The curve complex is instrumental in the proofs of many results about mapping class groups, including Nielsen–Thurston classification and the fact the the group is generated by finitely many Dehn twists. The curve complex is also the classical example of a hierachically hyperbolic space, and as such, provides a way of studying the mapping class group and specifically any behavior that does not come from the automorphisms of a proper subsurface \cite{Behrstock2017, Sisto2017}. The hope is to establish an analogous complex for Artin groups with similar properties and eventually to show that it has a similar hierarchical structure. 

In \cite{Cumplido2017}, Cumplido et al  consider a finite type Artin group $\AG$ and construct a simplicial complex, $\CPG$, called the \textit{complex of parabolic subgroups}, which generalizes the curve complex for braid groups. In \cite{CalvezWiest2019}, Calvez and Wiest relate this complex to other known complexes associated to finite type Artin groups and show that for irreducible finite type Artin groups, this complex has infinite diameter. 

In this paper, we extend the construction  of the complex of parabolic subgroups to FC type Artin groups. This involves proving a technical lemma that establishes the adjacency conditions for the vertices in the complex (See Lemma \ref{thrm: Define Complex}).

 We also show that in the FC type case, the complex of parabolic subgroups is connected with infinite diameter, conditional on some mild conditions on $\G$, the graph which encodes the relations between the generators in $S$. 
 
 \begin{lemma*}{\ref{lem_connected}}
Let $\AG$ be an irreducible FC type Artin group. The complex of parabolic subgroups is connected if and only if $\Gamma$ is connected and has at least three vertices.
\end{lemma*}

\begin{theorem*}{\ref{thrm:inf_diam}}
Suppose that $\AG$ is an irreducible FC type Artin group. Further suppose that the defining graph $\G$ is not a join and has at least two vertices. Then the complex of parabolic subgroups has infinite diameter. 
\end{theorem*}

The eventual goal is to show that the complex of parabolic subgroups is hyperbolic, which could lead to a hierarchically hyperbolic structure for Artin groups, similar to the curve complex for mapping class groups. There are other spaces associated to Artin groups which are known to be hyperbolic. For example, when the Deligne complex is CAT(0), its intersection graph will automatically be hyperbolic. However, the complex of parabolic subgroups is a complex where the stabilizers of the vertices are always finite type parabolic subgroups, that is the stabilizers are also Artin groups and so also have associated complexes. Other known hyperbolic complexes associated to a given Artin group do not exhibit this behavior and this is exactly the type of nesting structure required for constructing a hierarchically hyperbolic space. 


Section \ref{sec:Background and Notation} defines terms needed later on and introduces the Deligne complex, which will be the main geometric tool used in this paper. Section \ref{sec: Intersections} shows that the set of parabolic subgroups is closed under intersections. Section \ref{sec:complex} constructs the complex of parabolic subgroups. Finally, in Section \ref{sec:Infinite diameter}  we show that the complex of parabolic subgroups has is connected and infinite diameter(under mild assumptions). 




\section{Background on Artin groups and Deligne complexes}\label{sec:Background and Notation}

Given a finite simplicial  graph $\G$ with set of vertices $S=\{s_1\dots s_n\}$ and with an edge between $s_i$ and $s_j$ labeled by an integer $m_{ij}\geq 2$, we define the \emph{Artin group}, $\AG$,  associated to $\G$ to be the group with presentation
$$A=\langle s_1,\dots,s_n\ \mid\ 
\underbrace{s_is_js_i\dots}_{m_{ij}\text{ terms}}=\underbrace{s_js_is_j\dots}_{m_{ij}\text{ terms}}\ 
\textrm{ for all } i\neq j\rangle \,.$$

If there is no edge in $\Gamma$ between $s_i$ and $s_j$ we say that $m_{ij}=\infty$. The Coxeter group, $\WG$, associated to $\G$ is the group obtained from $\AG$ by adding the relations $s_i^2=1$ for all $i.$ If $\WG$ is a finite group, then we say that $\AG$ is finite type. 

For each subset $T\subset S$ we consider the subgroup generated by $T$. This is the Artin group corresponding to the full subgraph of $\G$ spanned by $T$ \cite{VanderLek1983} and is denoted by $A_T$. Such groups are known as \textit{standard parabolic subgroups} and more generally their conjugates are known as \textit{parabolic subgroups}. We say that a parabolic subgroup $gA_Tg^{-1}$ is \textit{finite type} if $T$ generates a finite type subgroup. 

An Artin group is \textit{irreducible} if it cannot be decomposed as a direct product of Artin groups. Am Artin group $\AG$ is reducible exactly  $S=S_1\sqcup S_2$ such that every $s_1\in S_1$ is commutes with every $s_2\in S_2$. In terms of the graph $\G$ reducible Artin groups correspond to joins of subgraphs where every edge between the subgraphs is labeled by 2. We will focus primarily on irreducible Artin groups.  

The moniod $\AG^+$ is the monoid generated by elements of $S$ and by \cite{Paris2002} it injects into the Artin group $\AG$. There is a prefix order on $\AG$ defined by $a\leq b$ if $a^{-1}b\in \AG^+$. When the Artin group $\AG$ is of finite type, the elements of the group form a lattice under this ordering \cite{Dehornoy1999}. 
This lattice is a part of the Garside structure for the group and many of the algebraic results about finite type Artin groups rely on this structure. (For more information about the Garside structure see \cite{DehornoyBook}.)
 For a given subset $T\subset S$, the group element $\Delta_T$ is defined to be the least common multiple of the generators $t\in T$ under the prefix order, and $\Delta_S$ is often denoted $\Delta.$

In a finite type Artin group $\AG$, either $\Delta$ or $\Delta^2$ is an element of the center of $\AG$ \cite{Dehornoy1999}, and if $\AG$ is irreducible then this element generates the center. Given a parabolic subgroup $P=gA_Tg^{-1}$ we define the group element $z_P$ to be the element $g\Delta^k g^{-1}$, where $k$ is the smallest integer such that this element is in the center. Note that if $A_T$ is not a finite type subgroup, then $\Delta_T$ may not exist. In particular, many infinite type Artin groups have trivial center \cite{Charney2018}. As a result this paper will focus on finite type subgroups within infinite type groups so that we may associate to each subgroup $P$ the element $z_P$.

One geometric construction often used to study Artin groups is the Deligne complex.
\begin{definition}
Consider the set 

$$ \mathcal{S}^f= \{ T \subseteq S \mid \textrm{$A_T$  is finite type} \}$$

 The \textit{Deligne complex} is the cube complex whose vertices are cosets $gA_T$, $T \in \mathcal{S}^f$  and for any pair $gA_T \subset gA_{T'}$, the interval $[gA_T, gA_{T'}]$ spans a cube of dimension $|T' \smallsetminus T|$. 
\end{definition}
 
The action of $\AG$ on its cosets induces an action of on the Deligne complex. The Artin group acts cocompactly and by isometries on this cube complex. The stabilizer of a vertex $gA_T$ is the parabolic subgroup $gA_Tg^{-1}$ so we can use the Deligne complex to study parabolic subgroups via its vertex stabilizers. We denote by $gA_\emptyset$ the cosets of the trivial subgroup. These cosets often figure prominently in arguments involving the Deligne complex as they have trivial stabilizer.

An Artin group is \textit{of FC type} if every subset of generators spanning a clique in the graph $\Gamma$ generates a finite type Artin group. FC stands for Flag Complex, as Artin groups of FC type were first introduced in \cite{CharneyDavis1995} where it is shown that FC type Artin groups are exactly those Artin groups whose Deligne complex has CAT(0) cubical metric. The Deligne complex is homotopy equivalent to the universal cover of the hyperplane complement for the associated Coxeter group and as such can be used to show that FC type Artin groups have many desirable properties \cite{Altobelli1998, CharneyDavis1995, Godelle2007}.

There are two main observations about the Deligne complex which we will use repeatedly throughout this paper.

\begin{remark} \label{remark: Delinge complex fixed points}
 The first is that if a point $x$ in the interior of a cube is fixed by the action of an element $g\in \AG$ then $g$ fixes the entire cube. Each vertex is associated with a coset of the subgroup $A_X$ for some $X\subset S$ and the action of the group takes this coset to another coset of $A_X$. Thus there are no elements of the group which act to permute the vertices of a cube. In particular this means that in the FC type case, if an element of the group $g\in \AG$ fixes two vertices $v_1$ and $v_2$, then $g$ must pointwise fix, not only the geodesic between $v_1$ and $v_2$ but also any minimal length edge path between $v_1$ and $v_2$.
\end{remark}

\begin{remark} \label{remark:Deligne hyperlanes}
The second observation is that every hyperplane is the translate of a hyperplane intersecting an edge between the vertex $A_\emptyset$ and a vertex of the form $A_{\{s\}}$ for some $s\in S$. We call such a hyperplane a \textit{hyperplane of type $s$}. The action of $\AG$ on the Deligne complex preserves hyperplane types. Furthermore,  hyperplanes of type $s$ can only intersect hyperplanes  type $t$ for $t$ in the link of $s$ in the defining graph $\G$. In particular if an element $g$ fixes a single point of a hyperplane  of type $s$, $H_s$, then $g$ must send $H_s$ to itself,  as the only hyperplane of type $s$ which intersects $H_s$ is $H_s$ itself. 
\end{remark} 


\section{Intersections of parabolic subgroups} \label{sec: Intersections}

In this section, we consider FC type Artin groups  and show that the intersection of two  finite type parabolic subgroups is a parabolic subgroup. This is a generalization of Theorem 9.5 in \cite{Cumplido2017}.

\begin{theorem}\label{Thrm:Intersections}
Suppose that $\AG$ is an FC type Artin group. Given two finite type parabolic subgroups $P$ and $Q$ the intersection $P\cap Q$ is also a finite type parabolic subgroup. 
\end{theorem}



\begin{proof}

Consider the Deligne complex, and consider two vertices $v_P$ and $v_Q$ such that the stabilizer of  $v_P$ is the finite type parabolic subgroup $P$  and likewise the stabilizer of $v_Q$ is  $Q$. These vertices are not uniquely determined but will always exist, and this proof does not depend on the choice of vertex. Consider the geodesic between $v_P$ and $v_Q$. Because the Deligne complex is CAT(0), this geodesic is unique and so any element of $P\cap Q$ pointwise fixes the geodesic.

 By remark \ref{remark: Delinge complex fixed points}, any minimal length edge path from $v_P$ to $v_Q$ is pointwise fixed by $P\cap Q$. 
We choose such an edge path between $v_P$ and $v_Q$ and call this path $\gamma.$
Each edge in this path is either a \textit{downward edge}, that is an edge starting at a vertex $gA_X$ and ending at a vertex $ghA_{X\smallsetminus\{x_0\}}$ for some $x_0\in X$ and $h\in A_X$, 
or it is an \textit{upward edge,} that is an edge starting at a vertex $gA_X$ and ending a vertex $gA_{X\cup\{y_0\}}$ for some $y_0\notin X$. On a given edge path through the Deligne complex, define a \textit{turning point} to be a vertex that is either proceeded by a downward edge and followed by an upward edge (\textit{a valley}) or is proceeded by an upward edge and followed by a downward edge (\textit{a peak}). If there is a sequence of upward edges between $v$ and $v'$, then the stabilizer of $v$ is a subset of the stabilizer of $v'$, thus be keeping track of the direction of the edges in this way we can determine inclusion relations between parabolic subgroups. 

Consider the sequence of turning points along $\gamma$ and label them $v_1,v_2\dots v_{n-1}$. Rename $v_P=v_0$ and $v_Q=v_n$, so that the sequence $v_0, v_1,\dots v_n$ forms a sequence of vertices on the given edge path. 

We will now show the desired result by induction on $n.$

In the base case, we assume that $n=0$. In this case $v_0=v_n$ and so $P=Q$ and $P\cap Q=P$ which is a finite type parabolic subgroup. 


Now assume that the fixed edge path between the given parabolic subgroups $P$ and $Q$ has exactly $n$ turning points. Define the subgroup $P_i$ to be the stabilizer of the turning point $v_i$, with $P_0=P$ and $P_n=Q$. 

If we consider the part of $\gamma$ between $v_0$ and $v_{n-1}$ we see that this is a minimal edge path with $n-1$ turning points between $v_0$ and $v_{n-1}$ and is therefore fixed by $P_0\cap P_{n-1}$ (again see Remark \ref{remark: Delinge complex fixed points}). Thus, by the induction hypothesis $P_0\cap P_{n-1}$ is a finite type parabolic subgroup. 

Next note that the vertex $v_{n-1}$ is on the edge path between $v_0$ and $v_n$ and so $v_{n-1}$ is fixed by $P_0\cap P_n$. This implies that $P_0\cap P_n\subset P_{n-1}$. Thus $P_0\cap P_n=P_0\cap P_{n-1}\cap P_n$. We would like to show that this intersection is a finite type parabolic subgroup. 

We now consider two cases. 

Case 1: The final segment in the fixed edge path (from $v_{n-1}$ to $v_n$) consists entirely of upward edges.
  
This implies that $P_{n-1}\subsetneq P_n$ . Thus $P_0\cap P_{n-1}=P_0\cap P_{n-1}\cap P_n=P_0\cap P_n$ is a finite type parabolic subgroup as desired. 
 
Case 2: The final segment of the fixed edge path consists entirely of downward edges. 

 Note that $P_{n-1}$ is a finite type parabolic subgroup and both $P_0\cap P_{n-1}$ and $P_n$ are finite type parabolic subgroups of $P_{n-1}$. Thus by Cumplido et al \cite{Cumplido2017}(Theorem 9.5) $P_0\cap P_{n-1}\cap P_n$ is a finite-type parabolic subgroup as desired. 
\end{proof}


\begin{remark}
This proof uses the Deligne complex to reduce to the case where $\AG$ is finite type, established in \cite{Cumplido2017}. Another complex similar to the Deligne complex is the clique-cube complex, defined by Charney and the author in \cite{Charney2018}. This complex is CAT(0) for all Artin groups. However in the clique-cube complex, the stabilizers of vertices correspond to parabolic subgroups $gA_Tg^{-1}$ where the generators in $T$ forms a clique in the defining graph $\Gamma$. 
 If we knew that within an Artin group generated by a single clique the set of parabolic subgroups was closed under intersection, then we could use this argument  and the clique-cube complex to show that for an arbitrary Artin group, the intersection of two clique type parabolic subgroups is parabolic. In this way, this result falls into a larger category of results, all of which state: If property A is true for any Artin group defined by a single clique then property A is true for any Artin group. See \cite{GodelleParis2012a} for more examples of properties that follow this pattern. 
\end{remark}

Here is a corollary of the above lemma. 

\begin{corollary}\label{thrm: Minimal_subgroup}
Suppose that $\AG$ is an FC type Artin group. 
Given an element $\alpha$ inside a finite type parabolic subgroup of $\AG$, there exists a unique minimal(by inclusion) finite type parabolic subgroup containing $\alpha$.
\end{corollary}

In the finite type case, Cumplido et al \cite{Cumplido2017} define the minimal parabolic subgroup corresponding to an element of the group. The Garside structure of a finite type Artin group provides a well-defined normal form for an element $\alpha$, and the support of such an element is the set of generators that appear in this normal form. The element $\alpha$ is contained in the subgroup $A_{supp(\alpha)}$ but this subgroup may not be minimal. Cumplido et al show that after suitable conjugation to a particular element $\alpha'=g^{-1}\alpha g$ the subgroup $P_\alpha=gA_{supp(\alpha')}g^{-1}$ will be the minimal subgroup containing $\alpha$ (\cite{Cumplido2017} Theorem 1.1).

Theorem \ref{Thrm:Intersections} allows us to extend this definition to any FC type Artin group.


\begin{proof} Given $\alpha$ in a finite type parabolic subgroup $P$ we define $P_\alpha$ to be the minimal parabolic subgroup of $P$ which contains $\alpha$ using the definition in \cite{Cumplido2017}. We must now show that if $\alpha$ is contained in two different finite type parabolic subgroups $P$ and $Q$ then $P_\alpha=Q_\alpha$ so the minimal subgroup does not depend on the parabolic subgroup containing $\alpha$. 

We know from Theorem \ref{Thrm:Intersections} that $R=P\cap Q$ is a parabolic subgroup containing $\alpha$. We also know that $R_\alpha\subset R \subset P$ so the fact that $P_\alpha$ is the minimal subgroup in $P$ containing $\alpha$ implies $P_\alpha\subset R_\alpha$. However $R_\alpha$ is the minimal subgroup of $R$ containing $\alpha$, and $P_\alpha\subset R$ therefore $R_\alpha\subset P_\alpha$. This implies that $R_\alpha=P_\alpha$. 

By an identical argument we get $R_\alpha=Q_\alpha$ and so $P_\alpha=Q_\alpha$ and the minimal subgroup containing the element $\alpha$ is independent of the finite type parabolic subgroup containing $\alpha$.
 \end{proof}


\section{The complex of parabolic subgroups} \label{sec:complex}
In the case of the mapping class group on the disc with $n$ punctures, the curve complex is defined to be the flag simplicial complex whose vertices are the non-degenerate simple closed curves in the disc, up to isotopy. Two curves are joined by an edge if, possibly after some isotopy, they are disjoint. The braid group  on $n$ strands is the mapping class group for the disc with $n$ punctures, and we would like to find an algebraic way to describe the curve complex so that it may be generalized to finite type Artin groups and to Artin groups in general. 

In \cite{Cumplido2017}, Cumplido et al show that there is a bijection $\phi$ between non-degenerate simple closed curves in the punctured disc and proper irreducible parabolic subgroups of the braid group. For a curve $C$, $\phi(C)$ is defined to be those automorphisms of the disc whose support is contained within $C.$ They also show that disjoint curves correspond to parabolic subgroups that satisfy one of the conditions below. 
Cumplido et al then define the complex of parabolic subgroups to be the flag complex whose vertices are proper irreducible parabolic subgroups where two subgroups are joined by an edge if they satisfy these conditions. 

\begin{definition} \cite{Cumplido2017}
 In any Artin group $\AG$, two parabolic subgroups $P$ and $Q$ are called \textit{adjacent} if one of the following conditions holds. 
 \begin{enumerate}
\item $P\subsetneq Q$ or  $Q\subsetneq P$
\item $P\cap Q=\{1\}$ and $pq=qp$ for all $p\in P$ and $q\in Q$
\end{enumerate}

We define $\CPG$, the \textit{complex of parabolic subgroups} to be the simplicial complex whose vertices are finite type proper irreducible parabolic subgroups, and where two vertices are connected by an edge if they are adjacent in the above sense. 
\end{definition}

While this definition works for any Artin group, having multiple adjacency conditions makes this definition cumbersome. Cumplido et al describe an equivalent condition for two subgroups to be adjacent in the context of finite type Artin groups and we extend this result to context of FC type Artin groups.

 First we need the following lemma.


\begin{lemma} \label{lem: conj by centers}
Suppose that $\AG$ is an FC type Artin group. Let $P$ and $Q$ be two finite type parabolic subgroups of $\AG$. Then, for every $g\in\AG,$ one has $g^{-1}P{g} = Q$ if and only if $g^{-1}z_Pg = z_Q$.
\end{lemma}

This result was shown in the finite type case by Cumplido (\cite{Cumplido2019} Lemma 33) and this proof also applies in the FC type case.  It relies heavily on Godelle's results about normalizers and ribbons in \cite{Godelle2003} and which Godelle later extends to Artin groups of FC type in \cite{Godelle2007}.  

We now establish the equivalency of several adjacency conditions. 

\begin{lemma}\label{thrm: Define Complex} 
Suppose that $\AG$ is an FC type Artin group, 
and $P$ and $Q$ are irreducible finite type parabolic subgroups. The following conditions are equivalent
\begin{enumerate}
\item P and Q are either equal or adjacent in the sense defined above. 
\item $z_Pz_Q=z_Qz_P$
\item There exists some $g\in \AG$ and $X,Y\subset S$ such that $P=gA_Xg^{-1}$, $Q=gA_Yg^{-1}$ and one of the following conditions holds 

\begin{enumerate}
\item $X\subset Y$ or  $Y\subset X$
\item $Y\cap X=\emptyset$ and $xy=yx$ for all $x\in X$ and $y\in Y$
\end{enumerate}

\end{enumerate}

Furthermore, if any of the conditions $1-3$ hold, then $P$ and $Q$ must both be contained in the same finite type parabolic subgroup. 
\end{lemma}

\begin{proof}
It is clear that $1$ implies $2$ and $3$ implies $1$. We will now show that $2$ implies $3$ by assuming that $z_Pz_Q=z_Qz_P$ and proving that $P$ and $Q$ are both contained the same finite type parabolic subgroup. This reduces the proof to the case where $\AG$ is finite type a case which was shown by Cumplido et al in (\cite{Cumplido2017} Theorem 2.2).

Consider the Deligne complex for the Artin group $\AG$, and suppose that $Fix(P)$ is the set of points that are pointwise fixed by the subgroup $P$. 
We would like to show that $Fix(P)$ and $Fix(Q)$ intersect in some vertex. This would imply that the stabilizer of that vertex is the desired finite type parabolic subgroup containing both $P$ and $Q$. 

Suppose for contradiction that $Fix(P)$ and $Fix(Q)$ are disjoint. By Remark \ref{remark: Delinge complex fixed points} both $Fix(P)$ and $Fix(Q)$ are convex subcomplexes, so if they are disjoint then there is some hyperplane $H$ which separates $Fix(P)$ and $Fix(Q)$. We choose $H$ as close as possible to the set $Fix(P)$. That is we would like some vertex $v_p\in Fix(P)$ such that $d(H,v_p)<1$. If $H$ does not have this property, for any point in $Fix(P)$ then Remark \ref{remark: Delinge complex fixed points} implies that we can replace $H$ with a different hyperplane closer to $Fix(P)$.

Now conjugate the entire complex so that $v_p=A_X$ for some $X\subset S$.  The hyperplane $H$ crosses at least one of the edges adjacent to $A_X$, and has $Fix(P)$ entirely on one side of it. Notice however that any vertex in the upward link of $A_X,$ that is any vertex of the form $A_{X\cup\{t\}}$ for some $t\in S$, must also be in $Fix(P)$ so $H$ divides $A_X$ from some vertex in the downward link of $A_X$. 
This vertex will have the form $gA_{X\smallsetminus\{s\}}$, for some $g\in A_X$ and some $s\in X$. In particular this shows that  after conjugation by $g^{-1}$, we have the hyperplane $H$ crossing the edge, $e_H$ between $A_X$ and $A_{X\smallsetminus\{s\}}$. Denote by $H^-$ the halfspace containing $A_{X\smallsetminus\{s\}}$ and $H^+$ the halfspace containing $A_X$.  By our assumption, the halfspace $H^-$ will contain $Fix(Q)$ and the halfspace $H^+$ with contain $Fix(P)$. 

\begin{figure}
\centering
\begin{tikzpicture}[thick, scale=.8]


\fill[pattern=north east lines, pattern color=gray!40](-6.4,2)--(-3.7,2)--(-3.7,-1.5)--(-6.4,-1.5);

\draw (-5.1,1) node{$Fix(Q)$};

\fill[pattern=north east lines, pattern color=gray!40](7,2)--(4.3,2)--(4.3,-1.5)--(7,-1.5);

\draw (5.7,1) node{$z_PFix(Q)$};

\fill[pattern=north east lines, pattern color=gray!40](-1,2)--(1,2)--(1,-1.5)--(-1,-1.5);

\draw (0,1) node{$Fix(P)$};

\draw (0,-.5)--(-3,-.5);
\draw (0,-.5)--(3,-.5);
\draw (0,-.5) node[vertex, label={[shift={(.2,0)}]$A_X$}]{};

\draw (-3,-.5) node[vertex, label={[shift={(.2,0)}]$A_{X\smallsetminus\{s\}}$}]{};

\draw (3,-.5) node[vertex, label={[shift={(0,0)}]$z_PA_{X\smallsetminus\{s\}}$}]{};


\draw[dashed,<->] (-1.5,2)--(-1.5,-1.5);
\draw[dashed,<->] (1.5,2)--(1.5,-1.5);

\draw (-1.5,2.2) node{$H$};

\draw(1.5,2.2) node{$z_PH$};

\end{tikzpicture}
\caption{ In the proof of Lemma \ref{thrm: Define Complex}, we assume that the hyperplane $H$ separates the convex sets $Fix(Q)$ and $Fix(P)$ and derive a contradiction by noting that $H$ and $z_pH$ do not cross, and yet $z_PFix(Q)\subset Fix(Q)$.}
\end{figure}
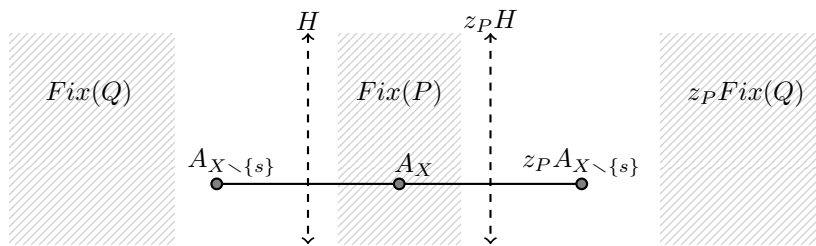

Now suppose that $v_q\in Fix(Q)$. 
Because $z_Pz_Q=z_Qz_P$, by Lemma \ref{lem: conj by centers} we know that $z_PQ=Qz_P$. This means that for any $q\in Q$ we have $qz_P=z_Pq'$ for some $q'\in Q$ and therefore $qz_Pv_q=z_Pq'v_q=z_Pv_q$.  Thus the vertex  $z_Pv_q$ must also be in $Fix(Q)$. However, we have assumed that that $Fix(Q)\subset H^-$.  We will show that $z_PFix(Q)\subset H^+$ giving a contradiction. 

The hyperplane $z_P H$ is a hyperplane intersecting the edge, $z_Pe_H$ between $z_PA_{X\smallsetminus\{s\}}$ and $z_PA_X=A_X$. The edge $z_Pe_H$ is adjacent to $A_X$, but not equal to $e_H$ as $A_{X\smallsetminus\{s\}}$ is in $H^-$ and so not in $Fix(P)$. This implies that $z_Pe_H$ is in $H^+$. Furthermore, by Remark \ref{remark:Deligne hyperlanes}, $H$ and $z_PH$ cannot intersect, therefore $z_PH^-\subset H^+$. This implies that $z_PFix(Q)\subset H^+$ giving a contradiction. 

Therefore $Fix(P)$ and $Fix(Q)$ must intersect in at least one vertex $v$ and this vertex has stabilizer $R$ which is some finite type parabolic subgroup. Both $P$ and $Q$ are subsets of $R$. This reduces the result to the case where both subgroups $P$ and $Q$ are contained in a finite type Artin group.
\end{proof}

The Artin group $\AG$ acts on the complex of parabolic subgroups, $\CPG$, by conjugation. This is an action by isometries as the adjacency conditions are preserved by conjugation. The action is also co-compact, with fundamental domain consisting of the subcomplex spanned by vertices corresponding to standard parabolic subgroups. 

The action is not proper as any element in a given parabolic subgroup will fix the vertex corresponding to this subgroup.

\section{Properties of the complex of parabolic subgroups}\label{sec:Infinite diameter}

The complex of parabolic subgroups is constructed so that, when restricted to the braid group case, it corresponds to the curve complex. We would like to know which properties of the curve complex carry over to the complex of parabolic subgroups. In particular we are interested in when the complex is connected, when it is finite or infinite diameter, when it is locally finite, as well as the question of when this complex might be hyperbolic. We will also consider the question of how, for some subset $T\subset S$, the complex of parabolic subgroups $\mathcal{CP}_{A_T}$ might relate to $\CPG$.

First we note that $\CPG$ gives very little information for reducible Artin groups. 

\begin{lemma} \label{lem:products}
 If $A_S$ is an FC type Arting groups which is a direct product $A_{S_1}\times A_{S_2}$ then $\CPG$ is the join of $\mathcal{CP}_{A_{S_1}}$ and $\mathcal{CP}_{A_{S_2}}$. In particular the complex has diameter $2$.
\end{lemma} 

\begin{proof}
 The set of vertices in $\CPG$ consists of only irreducible parabolic subgroups, so it is precisely the union of the set of proper irreducible parabolic subgroups in $A_{S_1}$ and the set of proper irreducible parabolic subgroups in $A_{S_2}$.  Any vertex representing a subgroup  $P< A_{S_1}$ must be connected by an edge to all the vertices representing subgroups $Q<{A_{S_2}}$ as in this case $z_Pz_Q=z_Qz_P$.
 \end{proof}
 
 Notice that this lemma applies even when $A_{S_1}$ or $A_{S_2}$ is not finite type. As a result of this lemma the rest of this paper will focus exclusively on irreducible Artin groups. 
 We now turn to the question of when the complex of parabolic subgroups is connected. 
 
\begin{lemma} \label{lem_connected}
Let $\AG$ be an irreducible FC type Artin group. The complex $\CPG$ is connected if and only if $\Gamma$ is connected and has at least three vertices. 
\end{lemma}

\begin{proof}

 If $\Gamma$ is a single edge with $m_{st}\geq 3 $ then $\AG$ (sometimes called a dihedral Artin group) has complex of parabolic subgroups that is totally disconnected. The vertex $A_{\{s\}}$ is not connected to any other vertices, as the element $s$ commutes only with powers of $s$ and with elements of the center of $\AG$, none of which generate the center of a proper subgroup distinct from $A_{\{s\}}$. 

 Similarly, if  $\AG$ can be decomposed as a free product, then $\CPG$ is not connected. In particular, if $\AG=A_{S_1}*A_{S_2}$ then any finite type parabolic subgroup must be contained in $A_{S_1}$ or $A_{S_2}$ or a conjugate of $A_{S_1}$ or $A_{S_2}$. The subcomplex consisting of the subgroups in $A_{S_1}$  is not connected to any other vertex as no element of $A_{S_1}$ commutes with an element not in $A_{S_1}$. In this way, we see that $\CPG$ is the disjoint union of subcomplexes corresponding to the conjugates of $A_{S_1}$ and $A_{S_2}$.

 
 
 However,  if $\Gamma$ is connected and has at least 3 vertices then $\CPG$ has connected fundamental domain. Any pair of vertices $s,t$ connected by an edge in $\G$ give either two connected vertices in $\CPG$, namely $A_{\{s\}}$ and $A_{\{t\}}$, or produce a vertex $A_{\{s,t\}}$ which is a proper irreducible parabolic subgroup and connected by edges to both $A_{\{s\}}$ and $A_{\{t\}}$. Thus as long as $\G$ is connected with more than $2$ vertices the fundamental domain of $\CPG$ is connected. 
 
 If $gA_Xg^{-1}$ and $A_Y$ are vertices on $\CPG$ then by writing $g$ as a sequence of normalizers $g=g_1g_2\dots g_k$ where $g_iA_{X_i}g_i^{-1}=A_{X_i}$, we obtain the following path in $\CPG$ from  $A_Y$ to $gA_Xg^{-1}$. 
 
 \[A_Y, A_{X_1}, g_1A_{X_2}g_1^{-1}, g_1g_2A_{X_3}(g_1g_2)^{-1}\dots gA_{X_{k-1}}g^{-1},gA_Xg^{-1}\]. 
 \end{proof}

This agrees with known facts about the curve complex. The curve complex of the mapping class group of a surface in connected except in the so called sporadic cases, that is cases where no essential curves in the surface are disjoint \cite{Masur2012}. 

The complex of parabolic subgroups is also similar to the curve complex in that, as long as it is connected, it is never locally finite. In particular, the vertex $A_{\{s,t\}}$ is connected to $A_{\{s\}}$, $A_{\{t\}}$, and all of their conjugates. If there are more than three vertices than $A_{\{s,t\}}$ is proper, and if $A_{\{s,t\}}$ is irreducible then are are infinitely many such conjugates. Thus $\CPG$ is not locally finite. 

Next we consider the question of when the complex of parabolic subgroups has infinite diameter. We have already established that the complex $\CPG$ has finite diameter when $\AG$ is reducible. The following theorem is a partial converse. 

\begin{theorem}\label{thrm:inf_diam}
Suppose that $\AG$ is an irreducible FC type Artin group. Further suppose that the defining graph $\G$ is not a join and has at least two vertices. Then the complex of parabolic subgroups $\CPG$ has infinite diameter. 
\end{theorem}

Note that in their paper \cite{CalvezWiest2019}, Calvez and Wiest show that the complex $\CPG$ has infinite diameter in the case where $\AG$ is finite type and irreducible. To do this they show that $\CPG$ is quasi isometric to the Cayley graph on the group $\AG$ constructed with the generating set 
\[\mathcal{N}=\{g\in \AG| gA_Xg^{-1}=A_X \text{ for some proper finite type irreducible $A_X$}\}\]

The set $\mathcal{N}$ contains the standard generating set $S$ and so is a generating set for the group. 
Using Lemma \ref{thrm: Define Complex}, Calvex and Wiest's proof that these spaces are quasi-isometric works in the FC type case. 

\begin{lemma} \label{lem:Quasi-Iso}
Suppose that $\AG$ is an irreducible FC type Artin group. Further suppose that $\CPG$ is connected.  Then the complex of parabolic subgroups, $\CPG$, and  $\Cayley$, the Cayley graph built with the generating set $\mathcal{N}$, are quasi-isometric. 
\end{lemma}


In order to prove Theorem \ref{thrm:inf_diam}, we also need the following technical lemma. This construction originally appeared in work by the author with Ruth Charney (see \cite{Charney2018} Lemmas 2.6 and 4.3). 

We say that an element $g$ act \textit{loxodromically} on a space if there is a bininfinite geodesic $\ell_g$ such that $g$ acts by translation along this geodesic. This geodesic is called the \textit{axis} of $g$. 

\begin{lemma} (\cite{Charney2018} Lemmas 2.6 and 4.3\label{lem: rank one})
Suppose that the defining graph of the Artin group $\G$ is a graph that is not a join and has at least two vertices. Then there is an element $g\in \AG$ and a constant $k$ such that

\begin{itemize}
\item  The element $g$ acts loxodromically on the Deligne complex and has axis $\ell_g$. 
\item No two hyperplanes which cross $\ell_g$ cross eachother. 
\item For any point $x \in \ell_g$, there are $2k$ hyperplanes crossing the segment of $\ell_g$ from $x$ to $gx$ and these hyperplanes cannot all be fixed by that same nontrivial element. 
\end{itemize}

\end{lemma}

We are now ready to prove Theorem \ref{thrm:inf_diam}. 
\begin{proof}
First we may assume that $\CPG$ is connected, as otherwise it will clearly have infinite diameter. If $\CPG$ is connected, by Lemma \ref{lem:Quasi-Iso} it is sufficient to show that $\Cayley$ has infinite diameter. 

 Let $g$  and $k$ be as in Lemma \ref{lem: rank one}. We will show that the path in $\Cayley$ from the identity to the element $g^n$ has length greater than $m=\frac{2kn}{2k+2}$. Thus by increasing $n$ we may obtain elements of arbitrarily large length, showing that the Cayley graph has infinite diameter.

Suppose for contradiction that there is a word $g^n=n_1\dots n_{m}$ where each $n_i$ is an element of $\mathcal{N}$ and there is some finite type irreducible $A_{X_i}$ such that $n_iA_{X_i}n_i^{-1}=A_{X_i}$. We will use the Deligne complex to derive a contradiction to this scenario showing that $g^n$ has word length greater than $m$ in the generating set $\mathcal{N}$.

In the Deligne complex consider the edge path, $\ell_g$ described in Lemma \ref{lem: rank one}. This is the axis for the loxodromic element $g$. 
\begin{figure}
\centering
\begin{tikzpicture}[ thick, scale=.9]


  						 \draw[<->] (-2,0)--(12,0);
    								
    					\draw	 (0,0)--(1,-1.5);
    					\draw (1,-1.5)--(2,-2.5);
    					\draw (3,-3)--(7,-3);
    					\draw	 (10,0)--(9,-1.5);
    					\draw (9,-1.5)--(8,-2.5);
    					
    					\draw (3,-3)--(3,-4);
    					\draw (3,-4)--(7,-4);
    					\draw (7,-4)--(7,-3);
    					
    					\draw(2.4,-2.75) node[label={[rotate=-30]$\dots$}]{};
    					
    					\draw(7.7,-2.8) node[label={[rotate=30]$\dots$}]{};
    						

		\draw (0,0) node[vertex,label={[shift={(0.05,0.05)}]$A_\emptyset$}]{};	
		\draw (1,-1.5) node[vertex,label={[shift={(-0.6,-0.2)}]$n_1A_\emptyset$}]{};		
		\draw (2,-2.5) node[vertex,label={[shift={(-.8,-.2)}]$n_1n_2A_\emptyset$}]{};	
		\draw (3,-3) node[vertex,label={[shift={(0.1,0)}]$hA_\emptyset$}]{};
		\draw (7,-3) node[vertex,label={[shift={(-0.2,0)}]$hn_{i}A_\emptyset$}]{};
		\draw (8,-2.5) node[vertex]{};
		\draw (9,-1.5) node[vertex,label={[shift={(1.3,-0.2)}]$n_1\dots n_{m-1}A_\emptyset$}]{};
		\draw (10,0) node[vertex,label={[shift={(0,0.05)}]$g^nA_\emptyset$}]{};
		
		\draw (3.9,0) node[vertex,label={[shift={(0,0.05)}]$x$}]{};
		\draw (5.75,0) node[vertex,label={[shift={(0,0.05)}]$gx$}]{};

		\draw (3,-4) node[vertex,label={[shift={(-0.2,-0.6)}]$hA_{X_i}$}]{};	
		\draw (7,-4) node[vertex,label={[shift={(.3,-0.6)}]$hn_{i}A_{X_i}$}]{};

		\draw (-2,0.3) 	node{\large $\ell_g$};
			 
	    \draw[dashed,<->] (4.2,1)--(4.2,-5); 
	      \draw[dashed,<->] (4.6,1)--(4.6,-5); 
	        \draw[dashed,<->] (5,1)--(5,-5); 
	          \draw[dashed,<->] (5.4,1)--(5.4,-5); 
	          
	          \draw[dashed, <->] (3.7,1)--(3.7,-2.5) arc (0:-90:1.3);
	          
	          \draw[dashed, <->] (6,1)--(6,-2.5) arc (180:270:1.2)--(8,-3.7);
		
		\draw [decorate,decoration={brace,amplitude=10pt},xshift=0pt,yshift=0pt]
(3.9,1) -- (5.75,1) node[shift={(1.5,.5)}]{$2k$ hyperplanes between $x$ and $gx$};			
				\end{tikzpicture}

\caption{In the proof of Theorem \ref{thrm:inf_diam}, there are at least $2k+2$ hyperplanes in a row crossing both $\ell_g$ and $[hA_\emptyset,hn_iA_\emptyset]$. This means there are at least $2k$ hyperplanes crossing $[hA_{X_i},hn_iA_{X_i}]$. }
\end{figure}
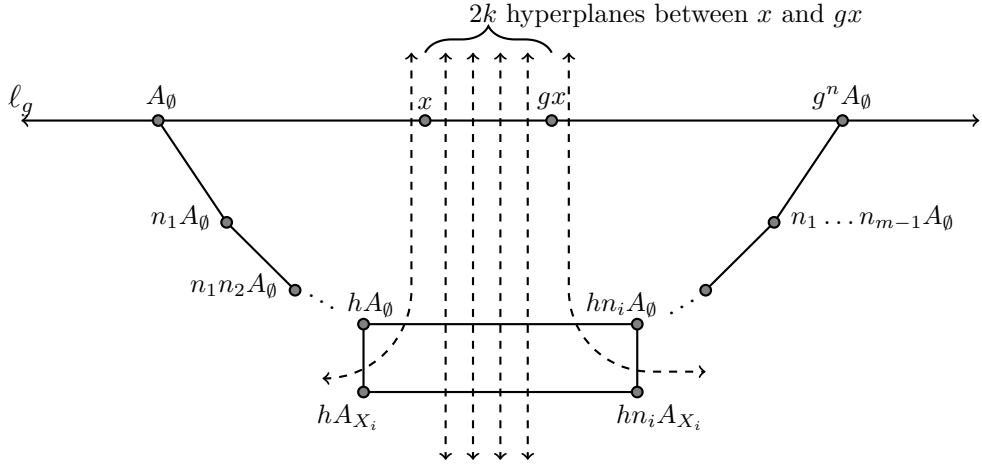
Now consider the geodesic polygon outlined by $\ell_g$ from $A_\emptyset$ to $g^nA_\emptyset$ and the geodesic segments  $[A_\emptyset, n_1A_\emptyset]$, $[n_1A_\emptyset,n_1n_2A_\emptyset],\dots, [n_1\dots n_{m-1}A_\emptyset,g^nA_\emptyset]$. The hyperplanes crossing $\ell_g$ all enter this polygon so they must exit this polygon, in the same order they entered, through one of the above geodesic segments.

There are $2kn$ hyperplanes crossing the segment $\ell_g$ between $A_\emptyset$ and $g^nA_\emptyset$ and there are $m$ segments each corresponding to $n_i$ for $1\leq i \leq m$ with $m=\frac{2kn}{2k+2}$ so there is some segment $[n_1n_2\dots n_{i-1}A_\emptyset, n_1n_2\dots n_{i}A_\emptyset]$ which is intersected by at least $2k+2$ of the hyperplanes intersecting $\ell_g$. 

Let $h=n_1\dots n_{i-1}$, and consider the geodesic rectangle formed by $hA_\emptyset,hA_{X_i},hn_iA_\emptyset,$ and $hn_iA_{X_i}$. Of the hyperplanes intersecting $\ell_g$, at least $2k+2$ of them enter this rectangle and at most two of them exit through the segments $[hA_\emptyset,hA_{X_i}]$ and $[hn_iA_\emptyset,hn_iA_{X_i}]$. This means, of the hyperplanes crossing $\ell_g$, there must be $2k$ in a row that intersect the geodesic segment $[hA_{X_i},hn_iA_{X_i}]$. 

The Deligne complex is CAT(0), and $n_i$ is a normalizer of $A_{X_i}$ so the geodesic segment $[hA_{X_i},hn_iA_{X_i}]$ is fixed pointwise by the parabolic subgroup $hA_{X_i}h^{-1}$. If a point on a hyperplane $H$ is fixed by an element of the group, then this element must take hyperplane $H$ to a hyperlane which intersects H, and thus must fix $H$ as a hyperplane (see Remark \ref{remark:Deligne hyperlanes}).   This means that any hyperplane crossing $[hA_{X_i},hn_iA_{X_i}]$ is also fixed as a hyperplane by $hA_{X_i}h^{-1}$. 

However, in Lemma \ref{lem: rank one}, we constructed $g$ so that any $2k$ consecutive  hyperplanes are only are fixed by the trivial element. This contradicts our assumption that $n_i$ was a normalizer. This  shows that this construction is impossible and $g^n$ has length greater than $m$ in $\Cayley$ as desired.
\end{proof}

We conjecture that $\CPG$ is infinite diameter for all irreducible FC type $\AG$. Calvez and Weist show this for all finite type Artin groups and we have shown this for all FC type Artin groups where the defining graph $\G$ is not a join. Thus it remains only to show this for irreducible infinite type Artin groups whose defining graph is a join. 

We also suspect that the complex is hyperbolic. 

\begin{conjecture}
If $\AG$ is an irreducible non-cyclic FC type Artin group, then $\CPG$ is an infinite diameter hyperbolic space. 
\end{conjecture}

If this conjecture were true it would be the first step towards establishing a hierarchically hyperbolic structure for Artin groups, similar to the curve complex for mapping class groups. This space is a good candidate for this kind of structure because of the way the complex $\mathcal{CP}_{A_T}$ is related to the complex $\CPG$ when $A_T$ is a subgroup of $\AG$. If $A_{T}\subsetneq \AG$ is an irreducible finite type subgroup then for any vertex $gPg^{-1}$ for $P$ some irreducible finite parabolic subgroup in $A_{T}$, $gPg^{-1}$ is connected by an edge to $gA_{T}g^{-1}$, so the copies of $\mathcal{CP}_{A_T}$ inside $\CPG$ are coned off. This means that, provided we could show $\CPG$ was hyperbolic, we could use $\CPG$ to isolate the behavior of the group that is not confined to a proper parabolic subgroup.

 However, showing that the space is hyperbolic is more difficult. This is the major challenge for research moving forward in this area.


\bibliographystyle{plain}
\bibliography{../Thesis/library}

\begin{thebibliography}{10}

\bibitem{Altobelli1998}
Joseph~A. Altobelli.
\newblock {The word problem for Artin groups of FC type}.
\newblock {\em Journal of Pure and Applied Algebra}, 129(1):1--22, 1998.

\bibitem{Behrstock2017}
Jason Behrstock, Mark~F Hagen, and Alessandro Sisto.
\newblock {Hierarchically hyperbolic spaces I: Curve complexes for cubical
  groups}.
\newblock {\em Geometry and Topology}, 21(3):1731--1804, 2017.

\bibitem{CalvezWiest2019}
Matthieu Calvez and Bert Wiest.
\newblock {Hyperbolic structures for Artin-Tits groups of spherical type}.
\newblock {\em arXiv:1904.02234}, 2019.

\bibitem{CharneyDavis1995}
Ruth Charney and Michael Davis.
\newblock {The K($\pi$,1)-problem for hyperplane complements associated to
  infinite reflection groups}.
\newblock {\em Journal of the American Mathematical Society}, 8(3):597--627,
  1995.

\bibitem{Charney2018}
Ruth Charney and Rose Morris-Wright.
\newblock {Artin groups of infinite type: trivial centers and acylindical
  hyperbolicity}.
\newblock {\em Proceedings of the American Mathematical Society}, 2019.

\bibitem{Cumplido2019}
Mar{\'{i}}a Cumplido.
\newblock {On loxodromic actions of Artin–Tits groups}.
\newblock {\em Journal of Pure and Applied Algebra}, 223(1):340--348, 2019.

\bibitem{Cumplido2017}
Mar{\'{i}}a Cumplido, Volker Gebhardt, Juan Gonz{\'{a}}lez-Meneses, and Bert
  Wiest.
\newblock {On parabolic subgroups of Artin-Tits groups of spherical type}.
\newblock {\em Advances in Mathematics}, to appear.

\bibitem{DehornoyBook}
Patrick Dehornoy, Francois Digne, Eddy Godelle, Daan Krammer, and Michel Jean.
\newblock {\em {Foundations of Garside Theory}}.
\newblock European Mathematical Society, Zurich, Switzerland, 2015.

\bibitem{Dehornoy1999}
Patrick Dehornoy and Luis Paris.
\newblock {Gaussian Groups and Garside Groups, Two Generalisations of Artin
  Groups}.
\newblock {\em Proceedings of the London Mathematical Society}, 79(3):569--604,
  1999.

\bibitem{Godelle2003}
Eddy Godelle.
\newblock {Parabolic subgroups of Artin groups of Type FC}.
\newblock {\em Pacific Journal of Mathematics}, 208(2), 2003.

\bibitem{Godelle2007}
Eddy Godelle.
\newblock {Artin-Tits groups with CAT (0) Deligne complex}.
\newblock {\em Journal of Pure and Applied Algebra}, 208(1):39--52, 2007.

\bibitem{GodelleParis2012a}
Eddy Godelle and Luis Paris.
\newblock {Basic questions on Artin-Tits groups}.
\newblock {\em Configuration Spaces}, pages 299--311, 2012.

\bibitem{Masur1998}
Howard~A. Masur and Yair Minsky.
\newblock {Geometry of the complex of curves I: hyperbolicity}.
\newblock {\em Inventiones Mathematicae}, 138(1):1--34, 1998.

\bibitem{Paris1997}
Luis Paris.
\newblock {Parabolic subgroups of Artin groups}.
\newblock {\em Journal of Algebra}, 196(2):369--399, 1997.

\bibitem{Paris2002}
Luis Paris.
\newblock {Artin monoids inject in their groups}.
\newblock {\em Commentarii Mathematici Helvetici}, 77(3):609--637, 2002.

\bibitem{Sisto2017}
Alessandro Sisto.
\newblock {What is a hierarchically hyperbolic space?}
\newblock {\em arXiv:1707.00053}, 2017.

\bibitem{VanderLek1983}
Harm van~der Lek.
\newblock {The homotopy type of complex hyperplane complements}.
\newblock {\em PhD thesis, Nijmegen}, 1983.

\end{thebibliography}

 \footnotesize

  R.~Morris-Wright, \textsc{Department of Mathematics, Brandeis University}\par\nopagebreak
  \textit{E-mail address}: \texttt{rmorriswright@brandeis.edu}

\end{document}